\documentclass[review]{elsarticle}

\usepackage{lineno,hyperref}

\usepackage{amsmath,amsthm,amsfonts,amssymb,amscd,amsxtra,mathrsfs,commath,url}
\usepackage{titlesec,booktabs}
\usepackage{epstopdf}
\usepackage{geometry,float,enumerate}
\usepackage{algorithm}%
\usepackage{algorithmicx}%
\usepackage{algpseudocode}%
\usepackage{makecell}
\usepackage{booktabs}
\usepackage{booktabs}
\usepackage{siunitx}
\usepackage{multirow}

\newsavebox\CBox

\titleformat*{\subsection}{\bfseries}

\journal{Journal of \LaTeX\ Templates}

\newtheorem{theorem}{Theorem}[section]

\newtheorem{lemma}[theorem]{Lemma}

\newtheorem{proposition}[theorem]{Proposition}

\newtheorem{assumption}[theorem]{Assumption}
\newtheorem{example}[theorem]{Example}
\newtheorem{remark}[theorem]{Remark}

\begin{document}
\begin{frontmatter}

\title{A Tight Surrogate-Based Adaptive Step Size  for the Multiobjective Conditional Gradient Method}

\author[mymainaddress]{Wang Chen}
\author[mysecondaddress]{Yong Zhao}
\author[mythirdaddress]{Liping Tang}
\author[mythirdaddress]{Xinmin Yang}

\cortext[mycorrespondingauthor]{ Email addresses: 
\href{mailto:chenwangff@163.com}{chenwangff@163.com} (Wang Chen),
\href{mailto:zhaoyongty@126.com}{zhaoyongty@126.com} (Yong Zhao),
\href{mailto:tanglipings@163.com}{tanglipings@163.com} (Liping Tang), \href{mailto:xmyang@cqnu.edu.cn}{xmyang@cqnu.edu.cn} (Xinmin Yang)
}

%
\address[mymainaddress]{School of Mathematical and Physical Sciences, Chongqing University of Science and Technology, Chongqing, 401331, China}
\address[mysecondaddress]{College of Mathematics and Statistics, Chongqing Jiaotong University, Chongqing 400074, China}
\address[mythirdaddress]{National Center for Applied Mathematics in Chongqing, Chongqing Normal University, Chongqing, 401331, China}

\begin{abstract}
	The multiobjective conditional gradient (MCondG) method is an effective descent-type algorithm for solving constrained multiobjective optimization problems. In this algorithm, the adaptive step size is derived by minimizing a one-dimensional convex quadratic surrogate function, which guarantees a sufficient decrease in all objective functions at each iteration.
	However, we observe that this surrogate function is relatively conservative, producing overly small step sizes and consequently slowing the convergence of the algorithm. This conservativeness stems from the fact that the surrogate function is dominated by the largest Lipschitz constant among the gradients of all objective functions. Motivated by this observation, we develop a tighter quadratic surrogate function that exploits the individual gradient Lipschitz constants. Based on this surrogate, we propose a new
	adaptive step size strategy obtained by evaluating this surrogate over a finite candidate set. We prove that
	the proposed step size preserves the descent property of the MCondG method, and has an improved worst-case complexity bound compared with the original adaptive step size strategy.
	Numerical experiments on benchmark and real-world problems demonstrate the effectiveness of the proposed strategy in terms of computational efficiency and solution quality.
\end{abstract}

\begin{keyword}
Multiobjective optimization, Conditional gradient method, Adaptive step size, Surrogate function, Convergence
\end{keyword}

\end{frontmatter}

\section{Introduction}

In this paper, we consider the constrained multiobjective optimization problem
\begin{equation}\label{mop}
	\begin{aligned}
		\min\quad &F(x)=\big(f_1(x),f_2(x),\cdots,f_m(x)\big)^{\top}\\
		\text{s.t.}\quad &x\in\Omega,
	\end{aligned}
\end{equation}
where $F:\mathbb{R}^n\rightarrow\mathbb{R}^m$ is continuously differentiable and $\Omega\subset\mathbb{R}^n$ is a nonempty compact convex set. A wide range of practical problems arising in engineering design, economics, management science, biomedicine and machine learning can be formulated within the framework of \eqref{mop}; see, for example, \cite{tanabe2020easy,kumar2021benchmark,zapotecas2023engineering}  and the references therein. 
Since the objective functions are generally conflicting, the solution concept is naturally characterized by Pareto optimality.

Several classes of methods have been developed for solving problem \eqref{mop}, including scalarization methods \cite{tang2021modified}, heuristic approaches \cite{chen2026decomposition}, and descent-based methods. Scalarization methods typically require preference information or appropriately selected scalarization parameters, which may depend on the user's preferences. Heuristic approaches, on the other hand, generally do not provide rigorous convergence guarantees.
Descent-based methods can overcome these limitations by extending classical descent algorithms for single-objective optimization to the multiobjective setting. Representative approaches include multiobjective projected gradient method \cite{DI2004}, multiobjective conditional gradient method
\cite{assunccao2021conditional}, multiobjective Zoutendijk method \cite{morovati2019extension}, multiobjective trust-region method, multiobjective reduced gradient methods \cite{moudden2018multiple} and multiobjective sequential quadratic programming method \cite{fliege2016method}. 

The multiobjective conditional gradient (MCondG) method, also known as multiobjective Frank--Wolfe method, has attracted considerable attention. Its appeal mainly stems from its simplicity and ease of implementation. 
In particular, MCondG method preserves feasibility by updating the iterate as a convex combination of the current iterate and a feasible solution, and avoids explicit projection computations by solving a linearized minimax subproblem over the constraint set at each iteration. The MCondG method was originally introduced by Assun{\c{c}}{\~a}o et al. in their seminal work \cite{assunccao2021conditional}, where its convergence properties were investigated under three different step size strategies, namely, the Armijo, adaptive and diminishing step sizes. Since then, a number of variants
and extensions of the MCondG framework have been developed in \cite{assunccao2024generalized,chen2023conditional,chen2024convergence,fan2025faster,gebrie2024adaptive,gonccalves2024away,gonccalves2025improved,li2023generalized,upadhayay2024nonmonotone,chen2025conditionalb}. For instance, Chen et al. \cite{chen2024convergence} investigated the
convergence behavior of MCondG with the adaptive step size when $\Omega$ is an unbounded closed convex set. Fan and Tang \cite{fan2025faster} established improved convergence rates for MCondG with the adaptive and diminishing step sizes under uniform convexity of the objective functions. Gon{\c{c}}alves
et al. \cite{gonccalves2024away} proposed an away-step variant of MCondG with
an adaptive step size for polyhedral constraint sets. More recently,
Gon{\c{c}}alves et al. \cite{gonccalves2025improved} established faster
sublinear and linear convergence rates for MCondG with the adaptive step size under different assumptions, including strong convexity of the objective
functions and uniform convexity of the constraint set. Upadhayay et al.
\cite{upadhayay2024nonmonotone} proposed a nonmonotone MCondG method in which
the step size is determined by a multiobjective average-type nonmonotone
line-search strategy introduced in \cite{mita2019nonmonotone}. Furthermore,
the MCondG framework has been extended to composite multiobjective
optimization problems \cite{assunccao2024generalized,gebrie2024adaptive,
	li2023generalized} and vector optimization problems
\cite{chen2023conditional}.

From the original MCondG method and its variants, the adaptive step size is particularly attractive because
it admits a closed-form expression while guaranteeing descent in all objective functions at every iteration (see \cite{assunccao2021conditional,assunccao2024generalized,chen2024convergence,fan2025faster,gonccalves2024away,gonccalves2025improved}). This strategy is based on the assumption that the gradient of each objective function $f_{i}$ is Lipschitz continuous with Lipschitz constant $L_{i}>0$, for $i\in\langle m\rangle=\{1,2,\cdots,m\}$; see Assumption \ref{lip_con}. By this assumption and Lemma 1 in \cite{assunccao2021conditional}, one has the following component-wise descent estimate:
\begin{equation}\label{des_ineq}
	\begin{aligned}
		f_{i}(x^{k}+td^{k})-f_{i}(x^{k})\leq \dfrac{1}{2}L_{\max}\|d(x^{k})\|^{2}t^{2}	+	 
		\theta(x^{k}) t,\quad t\in[0,1].
	\end{aligned}
\end{equation}
where $d^{k}$ denotes the search direction generated by the MCondG method, $L_{\max}=\max_{i\in\langle m\rangle}L_{i}$ and $\theta(x^{k})=\max_{i\in\langle m\rangle}\langle \nabla f_{i}(x^{k}),d^{k}\rangle$. The right-hand side of \eqref{des_ineq} defines a one-dimensional convex quadratic surrogate function. Minimizing this surrogate over $[0,1]$ gives the adaptive step size  (see \cite{assunccao2021conditional}):
\begin{equation}\label{stp_tk}
	t^{\rm AS}_{k}=\min\left\{1,\dfrac{-\theta(x^{k})}{L_{\max}\|d(x^{k})\|^{2}}\right\}.
\end{equation}

At first glance, the adaptive step size in \eqref{stp_tk} is appealing due to
its simple expression and low computational cost. Nevertheless, it is derived
from a worst-case quadratic surrogate that uses only the maximum Lipschitz
constant $L_{\max}$ to uniformly bound the curvature of all objective functions.
Such an aggregation may result in a loose upper bound when the individual
Lipschitz constants $L_1,L_2,\ldots,L_m$ vary significantly. Consequently, the
resulting surrogate can be overly conservative, leading to unnecessarily small
step sizes and, in turn, slowing down the convergence of the algorithm. The key limitation is not the use of a quadratic upper model itself, but the
additional relaxation that replaces the heterogeneous curvature information
$\{L_i\}_{i=1}^m$ by the single worst-case constant $L_{\max}$. Two illustrative examples are then provided to demonstrate the effect of this
relaxation and to quantify the conservativeness of the resulting step size.
This observation naturally leads to the following research question:

\begin{center}
	\emph{Can we develop a tighter one-dimensional convex surrogate  that more effectively exploits the individual Lipschitz constants $L_{1},L_{2},...,L_{m}$, rather than relying solely on $L_{\max}$, and thereby yields larger step sizes?}
\end{center}

In this paper, we provide an affirmative answer to this question by proposing a novel tight adaptive step size (TAS) strategy for MCondG methods. The key idea is to replace the conventional $L_{\max}$-based quadratic surrogate with a tighter piecewise-quadratic surrogate that explicitly incorporates the individual Lipschitz constants $L_1,L_{2}\ldots,L_m$. This new construction aims to reduce the conservativeness of the existing adaptive step size while preserving the descent property of MCondG. The main contributions of this paper are summarized as follows.

\begin{enumerate}[(i)]
	\item We identify a limitation of existing adaptive step size strategies based on $L_{\max}$: the use of a common worst-case curvature bound may lead to an overly loose convex quadratic surrogate and consequently to excessively small step sizes, thereby limiting the practical convergence efficiency of the MCondG method. We further illustrate this phenomenon through two numerical examples, providing a quantitative assessment of the conservativeness of the existing adaptive step size strategy.
		
	\item We develop a novel tight adaptive step size rule via direct minimization of a refined piecewise-quadratic surrogate function. Unlike the classical strategy based solely on $L_{\max}$, the new surrogate explicitly exploits the individual Lipschitz constants $L_1,L_{2}\ldots,L_m$. We further prove that this new step size can be computed exactly by evaluating the surrogate over a finite set of candidate points, resulting in an implementable procedure without requiring an additional line search.
	
	\item We establish the convergence properties of the MCondG method equipped wit	the proposed tight adaptive step size. In particular, we prove that the generated sequence of objective vectors is component-wise decreasing and that the algorithm converges to a Pareto stationary point. Moreover, we establish a tighter worst-case convergence bound compared with the corresponding result for the	adaptive step size in \cite{assunccao2021conditional}.

	\item We conduct numerical experiments on benchmark and real-world	multiobjective optimization problems. The numerical results demonstrate that	the proposed tight adaptive step size achieves better computational efficiency while maintaining comparable solution quality.
\end{enumerate}

The remainder of this paper is organized as follows. Section \ref{sec:2} recalls some preliminary results on multiobjective optimization and the MCondG method. Section \ref{new_stp} presents the proposed adaptive step size strategy. Section \ref{convergence} contains the convergence analysis of the algorithm.  Numerical experiments are reported in Section \ref{num_exp}. Finally, in Section \ref{conclusion}, some concluding remarks are given.

\section{Notations and Preliminaries}\label{sec:2}

Throughout this paper, $\langle\cdot,\cdot\rangle$ and $\|\cdot\|$ denote,
respectively, the standard inner product and the Euclidean norm in
$\mathbb{R}^n$. For a matrix $A\in\mathbb{R}^{n\times n}$, $\|A\|_2$ denotes its induced matrix 2-norm.
Let $e=(1,1,\ldots,1)^{\top}\in\mathbb{R}^{m}$. As usual, for $u,v\in\mathbb{R}^{m}$, we use ``$\preceq$'' to denote the partial order induced by the nonnegative orthant, i.e.,
\begin{equation*}
	u\preceq v ~\Leftrightarrow ~v-u\in\mathbb{R}_{+}^{m}.
\end{equation*}
 
In multiobjective optimization, the concept of optimality is replaced by that of Pareto optimality; see, e.g., \cite{miettinen1999nonlinear}. A point $\bar{x}\in\Omega$ is said to be Pareto optimal if there is no $x\in\Omega$ such that $F(x)\preceq F(\bar{x})$ and $F(x)\neq F(\bar{x})$. A point $\bar{x}\in\Omega$ is said to be weakly Pareto optimal if there is no $x\in\Omega$ such that $f_{i}(x)< f_{i}(\bar{x})$ for all $i\in\langle m\rangle$. A point $\bar{x}\in\Omega$ is called Pareto stationary (or Pareto critical) for \eqref{mop} if
\begin{equation*}
	\max_{i\in\langle m\rangle}~\langle \nabla f_{i}(\bar{x}),p-\bar{x}\rangle\geq0,\quad\forall p\in\Omega.
\end{equation*}
Under the convexity of the objective functions, every  Pareto critical point is a weakly Pareto optimal solution; see, e.g., \cite{miettinen1999nonlinear,assunccao2021conditional}. Since $\Omega$ is compact, its diameter $D=\max\{\|x-y\|:x,y\in\Omega\}$ is finite.

For a given point $x\in\Omega$, the search direction of the MCondG method is defined as 
\begin{equation}\label{des_direction}
	d(x)=p(x)-x,
\end{equation}
where
\begin{equation}\label{opt_sol}
	p(x)\in\mathop{\arg\min}_{p\in\Omega}\max_{i\in\langle m\rangle}~\langle \nabla f_{i}(x),p-x \rangle.
\end{equation}
Since $\Omega$ is compact, it follows that \eqref{opt_sol} has a solution, and thus $p(x)$ is well defined. The optimal value of \eqref{opt_sol} is denoted by $\theta(x)$, i.e.,
\begin{equation}\label{opt_val}
	\theta(x)=\max_{i\in\langle m\rangle}~\langle \nabla f_{i}(x),p(x)-x \rangle.
\end{equation}

In the following, we state several properties of the function $\theta(x)$; see \cite[Proposition 5]{assunccao2021conditional}.
\begin{lemma}\label{theta_property}
	The following statements hold:
	\begin{enumerate}[\rm(i)]
		\item $\theta(x)\leq0$ for all $x\in\Omega$\label{negative};
		\item $\theta(x)=0$ if and only if $x\in\Omega$ is a Pareto critical point of \eqref{mop};
		\item $\theta(x)$ is continuous.\label{theta_continuous}
	\end{enumerate}
	Moreover, if $\theta(x)<0$, then the direction $d(x)$ satisfies $\langle\nabla f_i(x),d(x)\rangle\leq\theta(x)<0$ for
	$ i\in\langle m\rangle$.
	Thus, $d(x)$ is a common descent direction for all objective functions at $x$.
	
\end{lemma}

In the following, we recall the MCondG method proposed in \cite{assunccao2021conditional}. 

\begin{algorithm}[H]
	\caption{MCondG algorithm}
	\label{alg:MCondG}
	\begin{algorithmic}[1]
		
		\Require Initial point $x^0\in\Omega$.
		\Ensure A Pareto stationary point.
		
		\State Compute $p(x^0)$ and $\theta(x^0)$.
		\State Set $k\gets0$.
		
		\While{$\theta(x^k)\neq0$}
		
		\State Compute the search direction
		\[
		d^k=p(x^k)-x^k.
		\]
		
		\State Compute a step size $t_k\in(0,1]$ using a prescribed step size strategy.
		
		\State Update
		\[
		x^{k+1}=x^k+t_kd^k.
		\]
		
		\State Compute $p(x^{k+1})$ and $\theta(x^{k+1})$.
		
		\State $k\gets k+1$.
		
		\EndWhile
		
	\end{algorithmic}
\end{algorithm}

For convenience, we use $d^{k}$ to denote $d(x^{k})$.
By Lemma \ref{theta_property}, the MCondG algorithm terminates whenever a Pareto critical point is reached. In the convergence analysis below, we consider the non-terminating case and assume that $\theta(x^{k}) < 0$ for all $k\geq0$. Hence, the algorithm generates an infinite sequence $\{x^k\}\subset\Omega$ (see \cite{assunccao2021conditional}), and  $d^{k}$ is a descent direction. We end this section by stating a key assumption that will be useful in our subsequent analysis of the step size.
 
 \begin{assumption}\label{lip_con}
 	For $i\in\langle m\rangle$, each objective function $f_{i}$ has a Lipschitz continuous gradient with Lipschitz constant $L_{i}>0$ , i.e.,
 	 $\|\nabla f_{i}(x)- \nabla f_{i}(y)\|\leq L_{i}\|x-y\|$ for all $x,y\in\Omega$.
 \end{assumption}
 
 \begin{remark}
 	Assumption \ref{lip_con} is a standard smoothness condition in multiobjective optimization and is satisfied by a broad class of commonly considered objective functions. In particular, consider the multiobjective convex quadratic problem studied, for example, in \cite{oberdieck2016objective,jahangiri2024solving}, where
 	\[
 	f_i(x)=\frac{1}{2}x^{\top}Q_i x+c_i^{\top}x+d_i,
 	\qquad i\in\langle m\rangle,
 	\]
 	where each $Q_i\in\mathbb{R}^{n\times n}$ is symmetric positive semidefinite, $c_{i}\in \mathbb{R}^{n}$ and $d_{i}\in\mathbb{R}$.
 	Then $\nabla f_i(x)=Q_i x+c_i$,
 	and hence $L_i=\|Q_i\|_2$ for $i\in\langle m\rangle$.
 \end{remark} 

\section{Derivation of the new step size}\label{new_stp}

In this section, we first revisit the derivation of the adaptive step size proposed in \cite{assunccao2021conditional} and illustrate its potential limitation through two examples. Motivated by these observations, we introduce a new adaptive step size obtained by directly minimizing a tighter quadratic model. 

\subsection{Limitations of the adaptive step size}

Using Assumption \ref{lip_con} and \cite[Lemma 5.7]{beck2017first}, we obtain the following component-wise descent inequality for each objective $i\in\langle m\rangle$:
\begin{equation*}
	f_{i}(y)\leq f_{i}(x)+\langle \nabla f_{i}(x),y-x\rangle+\dfrac{1}{2}L_{i}\|x-y\|^{2},\quad \forall x,y\in\Omega.
\end{equation*}
Let $y=x^{k}+td^{k}$ with $t\in[0,1]$ and $x=x^{k}$. Combining the above inequality with the relation \eqref{opt_val} and $L_{\max}=\max_{i\in\langle m\rangle} L_{i}$, for all  $i\in\langle m\rangle$, we derive the following chain of relaxations:
\begin{equation}\label{descent-lemma}
	\begin{aligned}
		f_{i}(x^{k}+td^{k})-f_{i}(x^{k})&\leq \dfrac{1}{2}L_{i}\|d^{k}\|^{2}t^{2}+\langle\nabla f_{i}(x^{k}),d^{k}\rangle t\\
		&\leq \max_{i\in\langle m\rangle}\left\{\dfrac{1}{2}L_{i}\|d^{k}\|^{2}t^{2}+\langle\nabla f_{i}(x^{k}),d^{k}\rangle t\right\}\\
		&\leq \max_{i\in\langle m\rangle} \left\{\dfrac{1}{2}L_{i}\|d^{k}\|^{2}t^{2}\right\}	+	 
		\max_{i\in\langle m\rangle}\{\langle \nabla F_{i}(x^{k}),d^{k}\rangle t\}\\
		&=\dfrac{1}{2} L_{\max}\|d^{k}\|^{2}t^{2}	+	 
		\theta(x^{k}) t,\quad \forall t\in[0,1].
	\end{aligned}
\end{equation}

To facilitate the subsequent analysis, we define three auxiliary functions:
\begin{align}
	g_{i,k}(t)&= \dfrac{1}{2}L_{i}\|d^{k}\|^{2}t^{2}+\langle\nabla f_{i}(x^{k}),d^{k}\rangle t,\quad i\in\langle m\rangle,\label{gki}\\
	g_{k}(t)&=\max_{i\in\langle m\rangle}g_{i,k}(t),\label{gk}\\
	h_{k}(t)&=\dfrac{1}{2} L_{\max}\|d^{k}\|^{2}t^{2}	+\theta(x^{k}) t.\label{hk}
\end{align}
Clearly, by \eqref{descent-lemma}, for each $i\in\langle m\rangle$, it holds that
\begin{align}\label{gki_gk_hk}
	g_{i,k}(t)\leq g_{k}(t)\leq h_{k}(t),\quad \forall t\in[0,1].
\end{align}
Observe that $h_{k}(t)$ depends on the largest Lipschitz constant $L_{\max}$ of the gradients over all objective functions, and moreover,
it serves as a loose upper bound for $F(x^{k}+td^{k}) - F(x^{k})$. The adaptive step size \eqref{stp_tk} proposed in \cite{assunccao2021conditional} is defined as the minimizer of $h_{k}(t)$ over $[0,1]$, i.e.,
\begin{equation*}
	t^{\rm AS}_{k}=\min\left\{1,\dfrac{-\theta(x^{k})}{L_{\max}\|d^{k}\|^{2}}\right\}=\underset{t\in[0,1]}{\arg\min}~ h_{k}(t).
\end{equation*}
We split the closed-form minimizer into two cases:
\begin{itemize}
	\item When $t^{\rm AS}_{k}=1$, we have
	\begin{align*}
		h_{k}(t_{k}^{\rm AS})=\dfrac{1}{2} L_{\max}\|d^{k}\|^{2}	+\theta(x^{k})\leq \dfrac{1}{2}\theta(x^{k}).
	\end{align*}
	\item When $t^{\rm AS}_{k}=-\theta(x^{k})/(L_{\max}\|d^{k}\|^{2})<1$, we get
	\begin{align*}
		h_{k}(t_{k}^{\rm AS})= -\frac{\theta(x^{k})^{2}}{2L_{\max} \|d^{k}\|^{2}}.
	\end{align*}
\end{itemize}
Hence, we arrive at an important  descent inequality that measures the improvement of the functional values sequence $\{F(x^{k})\}$ generated by the MCondG algorithm with the adaptive step size \eqref{stp_tk}, which was first established in Proposition 13 of \cite{assunccao2021conditional}. For completeness, we restate this property as follows:
\begin{proposition}\label{tas_des}
	Consider the MCondG algorithm with the adaptive step size \eqref{stp_tk}. Then,
	\begin{align*}\label{eq:monotone0}
		F(x^{k}+t_{k}^{\rm AS}d^{k})- F(x^{k})
		\preceq-\frac{1}{2}\min\left\{-\theta(x^{k}),\frac{\theta(x^{k})^{2}}{L_{\max} D^{2}}\right\}e.
	\end{align*}
\end{proposition}

Although the step size $t^{\rm AS}_{k}$ admits a closed-form expression and guarantees monotonic descent of the vector-valued objective sequence $\{F(x^{k})\}$, its formulation relies heavily on the maximum Lipschitz constant $L_{\max}$. If we focus solely on the first inequality in Lemma \ref{descent-lemma} and minimize its right-hand side quadratic surrogate over $[0,1]$, then we obtain an adaptive step size for each objective function, given by
\begin{equation}\label{stp_asi}
	t_{i,k}^{\rm AS} =  \min\left\{1, -\frac{\langle\nabla f_i(x^k),d^k\rangle}{L_i\|d^k\|^2}\right\}=\underset{t\in[0,1]}{\arg\min} \left\{\dfrac{1}{2}L_{i}\|d^{k}\|^{2}t^{2}+\langle\nabla f_{i}(x^{k}),d^{k}\rangle t\right\}, \quad i\in\langle m\rangle.
\end{equation}
These quantities $t_{i,k}^{\rm AS}$ with $i\in\langle m\rangle$ represent the ``ideal'' step size for each objective when optimized in isolation. However, none of the individual step sizes $t_{i,k}^{\rm AS}$ can  guarantee $F(x^{k}+t_{i,k}^{\rm AS}d^{k}) \preceq F(x^{k})$. Since each $t_{i,k}^{\rm AS}$ is computed using the objective-specific Lipschitz constant $L_{i}$, it is no smaller than $t_{k}^{\rm AS}$.

\begin{proposition}
	For all $k\geq0$, there holds
	\begin{equation*}
		t_{k}^{\rm AS}\leq t_{i,k}^{\rm AS},\quad\forall i\in\langle m\rangle.
	\end{equation*}
\end{proposition}

\begin{proof}
	From $L_{i}\leq L_{\max}$ and $\langle\nabla f_i(x^k),d^k\rangle\leq\theta(x^{k})<0$ for all $i\in\langle m\rangle$, it follows that
	\begin{equation} 
		-\dfrac{\theta(x^{k})}{L_{\max}\|d^{k}\|^{2}}\leq-\frac{\langle\nabla f_i(x^k),d^k\rangle}{L_i\|d^k\|^2},\quad \forall i\in\langle m\rangle.
	\end{equation}
	Together with the definitions of $t_{k}^{\rm AS}$ in \eqref{stp_tk} and $t_{i,k}^{\rm AS}$ in \eqref{stp_asi}, this inequality implies $t_{k}^{\rm AS}\leq t_{i,k}^{\rm AS}$ for all $i\in\langle m\rangle$, as desired.
\end{proof}

The above proposition indicates that the adaptive step size $t_k^{\rm AS}$ is always bounded by the objective-wise step sizes $t_{i,k}^{\rm AS}$ due to the use of the maximum Lipschitz constant $L_{\max}$. This observation suggests that the adaptive step size in \cite{assunccao2021conditional} may be overly conservative, especially when the Lipschitz constants of different objective functions differ significantly. To illustrate this issue more clearly, we provide two examples in which the quadratic upper bound $h_k(t)$ and the individual quadratic models $g_{i,k}(t)$ are compared. The results demonstrate the gap between the conservative step size $t_k^{\rm AS}$ and the objective-wise ideal step sizes $t_{i,k}^{\rm AS}$.

\begin{example}\label{ex1}
	Consider the standard DGO2 test problem proposed in \cite{huband2006review}. We construct its modified variant, denoted as DGO2*, with $m=2$, $n=1$ and $\Omega=[-5,5]$, whose objective functions are defined as
	\begin{equation*}
		f_{1}(x)=(x-2)^{2},\qquad
		f_{2}(x)=9-\sqrt{81-x^{2}}.
	\end{equation*}
	Clearly, $L_{1}=2$ and $L_{2}=81/56^{1.5}$. Take $x^{0}=-4.9$. Then, by solving the associated problem \eqref{opt_sol}, we obtain  $\theta(x^{0})=-6.4259$ and $p(x^{0})= 5$. Thus, by \eqref{stp_tk}, we derive $t_{0}^{\rm AS}= 0.0327$, which is indicated  by  the red point in Figure \ref{fig:ex3.1}. In contrast, the individual minimizers are $t_{1,0}= 0.697$ and $t_{2,0}= 0.3392$, which are marked by the blue and black points, respectively.
	\begin{figure}[H]
		\centering
		\includegraphics[width=0.5\linewidth]{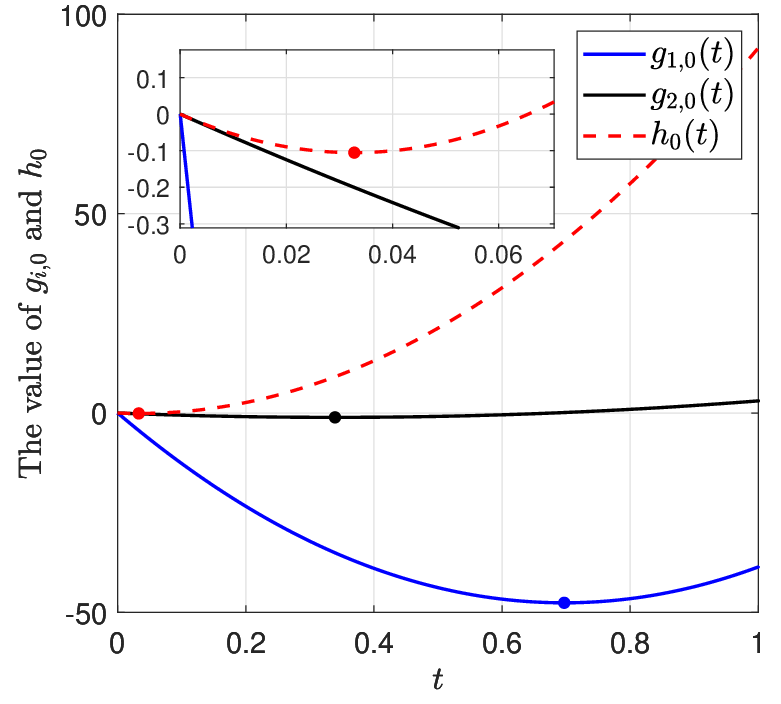}
		\caption{Graphs of $g_{1,0}(t)$, $g_{2,0}(t)$ and $h_{0}(t)$.}
		\label{fig:ex3.1}
	\end{figure}
\end{example}

\begin{example}\label{ex2}
	Consider the problem \eqref{mop} with $m=3$, $n=3$, $\Omega=[-1,1]^{3}$,
		\begin{equation*}
		f_{1}(x)=(2x_{1}-1)^{2}, \qquad
		f_{2}(x)=2(2x_{1}-x_{2})^{2}, \qquad
		f_{3}(x)=3(2x_{2}-x_{3})^{2}.
		\end{equation*}
	This instance is called Toi8, which can be found in \cite{mita2019nonmonotone,toint1983test}. For this problem, we have $L_{1}=8$, $L_{2}=20$ and $L_{3}=30$. Take $x^{0}=(-0.7, 0.6, -0.9)^{\top}$. Then, by solving the associated problem \eqref{opt_sol}, we get  $\theta(x^{0}) = -16.32$ and $p(x^{0})=(1,-1,1)^{\top}$. Thus, from \eqref{stp_tk}, we derive $t_{0}^{\rm AS} = 0.06$, which is marked by the red point in Figure \ref{fig:ex3.2}. The individual minimizers are $t_{1,0}= 0.2252$, $t_{2,0}= 0.2208$ and $t_{3,0}= 0.2364$, which are indicated by the blue, black and green points, respectively.
	\begin{figure}[H]
		\centering
		\includegraphics[width=0.5\linewidth]{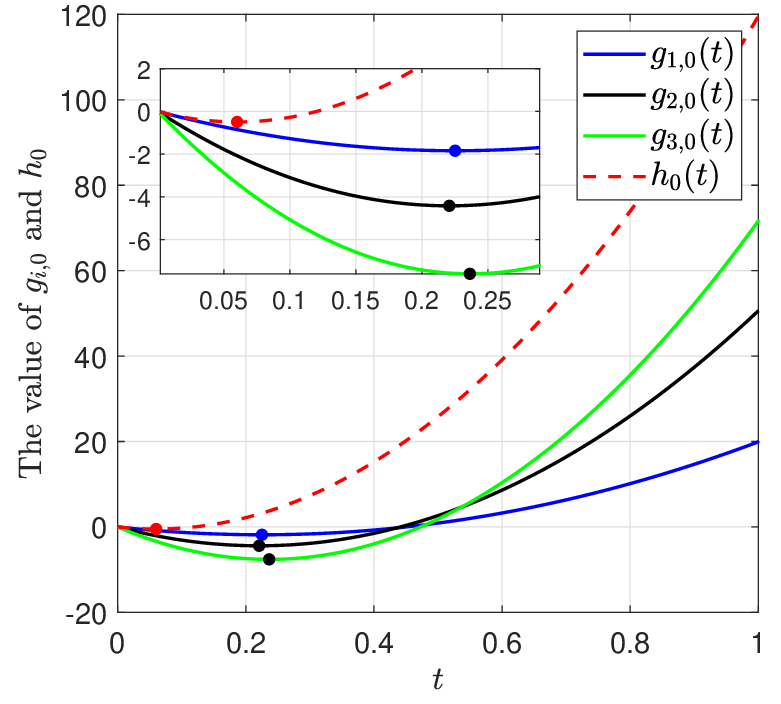}
		\caption{Graphs of $g_{1,0}(t)$, $g_{2,0}(t)$, $g_{3,0}(t)$ and $h_{0}(t)$.}
		\label{fig:ex3.2}
	\end{figure}
	
\end{example}

We conclude this section by characterizing the relationship between $t_k^{\rm TAS}$ and $t_k^{\rm AS}$ based on the relative values of the Lipschitz constants.

\begin{itemize}
	\item If $L_i=L_{\max}$ for all $i\in\langle m\rangle$, then $g_k(t)= h_k(t)$, and hence $t_k^{\rm TAS}=t_k^{\rm AS}$.
	
	\item If there exist $i,j\in\langle m\rangle$ such that $L_i\neq L_j$, then $g_k(t)\leq h_k(t)$ with strict inequality for some $t$, and $t_k^{\rm TAS}$ may be strictly larger than $t_k^{\rm AS}$.
\end{itemize}

\subsection{Construction of the Tight Adaptive Stepsize}

The preceding analysis and numerical examples show that the adaptive
step size $t_k^{\rm AS}$ may be overly conservative because it is derived
from the relaxed upper model $h_k$, which employs the largest Lipschitz
constant $L_{\max}$. In contrast, the
function $g_k$ retains the objective-wise Lipschitz constants
$L_i$ and therefore provides a tighter upper model. This motivates us
to determine the step size directly by minimizing $g_k$ over $[0,1]$, i.e.,
\begin{equation}\label{tight_ada_pro}
	\min_{t\in[0,1]}\; g_k(t).
\end{equation}

The following proposition establishes the basic properties of the
surrogate function $g_k$ and guarantees the uniqueness of the proposed
step size.

\begin{proposition}
	Let $g_{k}$ be as in \eqref{gk}. Then $g_{k}$ is strictly convex on $[0,1]$. Consequently, the problem \eqref{tight_ada_pro} admits a unique optimal solution.
\end{proposition}

\begin{proof}
	The proof is trivial, so we omit it for brevity.
\end{proof}

Denote by $t^{\rm TAS}_{k}$ the optimal solution of \eqref{tight_ada_pro} given by
\begin{equation}\label{exact_ada_pro1}
	t^{\rm TAS}_{k}=\underset{t\in[0,1]}{\arg\min}~ g_{k}(t)	
\end{equation}

The following observation clarifies the relationship between the
proposed tight adaptive stepsize and the classical adaptive stepsize.

\begin{remark}\label{rem:TAS_AS}
	If all objective functions have the same Lipschitz constant, i.e.,
	$L_i=L_{\max}$ for all $i\in\langle m\rangle$, then
	\[
	g_k(t)
	=
	\frac{1}{2}L_{\max}\|d^k\|^2t^2
	+
	\max_{i\in\langle m\rangle}
	\langle\nabla f_i(x^k),d^k\rangle t
	=
	h_k(t),
	\]
	for all $t\in[0,1]$. Consequently,
	\[
	t_k^{\rm TAS}=t_k^{\rm AS}.
	\]
	When the Lipschitz constants are not identical, the inequality
	$g_k(t)\leq h_k(t)$ may be strict for some $t\in[0,1]$. In this case,
	the proposed step size may differ from the classical adaptive step size,
	and can be strictly larger in suitable situations.
\end{remark}

To obtain the step size $t^{\rm TAS}_{k}$, we can adopt general univariate convex solvers, for example the built-in routine \texttt{fminbnd} in MATLAB. However, such solvers only yield approximate solutions and, more importantly, the resulting approximate step sizes are incompatible with the theoretical convergence analysis of our algorithm. Interestingly, the special structure of $g_k$ allows its exact minimizer to be obtained by evaluating finitely many candidate points. For this purpose, define 
\begin{align*}
	M_{k}&=\left\{-\frac{\langle\nabla f_{i}(x^{k}),d^{k}\rangle}{L_{i}\|d^{k}\|^{2}}:i\in\langle m\rangle\right\}\\
	I_{k}&=\left\{-\frac{2\langle\nabla f_{i}(x^{k})-\nabla f_{j}(x^{k}),d^{k}\rangle}{(L_{i}-L_{j})\|d^{k}\|^{2}}:i,j\in\langle m\rangle,\;i<j,\;L_{i}\neq L_{j}\right\}
\end{align*}
and the candidate set
\begin{align*}
	\mathcal{C}_{k}=\{1\}\cup M_{k} \cup I_{k}.
\end{align*}
The set $M_{k}$ includes the minimizers of the
individual quadratic functions $g_{i,k}$, while $I_{k}$
contains the pairwise intersection points of quadratic functions with
different quadratic coefficients. The next proposition shows that these
points are sufficient to determine the minimizer of $g_k$.	

\begin{proposition}\label{prop:3.4}
	The unique minimizer $t_k^{\rm TAS}$ of $g_k$ over $[0,1]$ satisfies
	\[
	t_k^{\rm TAS}\in\mathcal{C}_k\cap[0,1].
	\]
\end{proposition}

\begin{proof}
	First, we show $t_k^{\rm TAS}>0$. Since $g_k(0)=0$ and
	\begin{align*}
		\lim_{t\to0^+}\frac{g_k(t)-g_k(0)}{t}
		&= \lim_{t\to0^+}\max_{i\in\langle m\rangle}\left\{\frac{1}{2}L_i\|d^k\|^2t+\langle\nabla f_{i}(x^{k}),d^{k}\rangle \right\}\\
		&= \max_{i\in\langle m\rangle}\lim_{t\to0^+}\left\{\frac{1}{2}L_i\|d^k\|^2t+\langle\nabla f_{i}(x^{k}),d^{k}\rangle \right\}\\
		&= \max_{i\in\langle m\rangle}~ \langle\nabla f_{i}(x^{k}),d^{k}\rangle \\
		&= \theta(x^k)<0,
	\end{align*}
	there exists $\varepsilon>0$ such that $g_k(t)<g_k(0)$ for all $t\in(0,\varepsilon)$. Hence, $t_k^{\rm TAS}>0$.
	
	If $t_k^{\rm TAS}=1$, then clearly
	$t_k^{\rm TAS}\in\mathcal{C}_k$. It remains to consider the case $t_k^{\rm TAS}\in(0,1)$. Define the active index set
	\[
	\mathcal{A}_k
	=
	\left\{
	i\in\langle m\rangle:
	g_{i,k}(t_k^{\rm TAS})=g_k(t_k^{\rm TAS})
	\right\}.
	\]
	
	We next distinguish two cases:
	
	\emph{Case 1.} $|\mathcal{A}_k|=1$. Let $\mathcal{A}_k=\{i^*\}$. Then $t_k^{\rm TAS}$ is the unique minimizer of $g_{i^*,k}$, and hence
	\begin{align*} 
		t_{k}^{\rm TAS}=-\frac{\langle\nabla f_{i^{*}}(x^{k}),d^{k}\rangle}{L_{i^{*}}\|d^{k}\|^{2}}\in\mathcal{C}_{k}.
	\end{align*}
	
	\emph{Case 2.} $|\mathcal{A}_k|\geq 2$. Take $i^{*},j^{*}\in\mathcal{A}_k$ with $i^*\neq j^*$. Then, 
	\begin{align*}
		g_{i^{*},k}(t^{\rm TAS}_{k})=g_{j^{*},k}(t^{\rm TAS}_{k})=g_{k}(t^{\rm TAS}_{k}),
	\end{align*}
	which yields
	\begin{align*}
		\dfrac{1}{2}L_{i^{*}}\|d^{k}\|^{2}(t^{\rm TAS}_{k})^{2}+\langle\nabla f_{i^{*}}(x^{k}),d^{k}\rangle t^{\rm TAS}_{k}=\dfrac{1}{2}L_{j^{*}}\|d^{k}\|^{2}(t^{\rm TAS}_{k})^{2}+\langle\nabla f_{j^{*}}(x^{k}),d^{k}\rangle t^{\rm TAS}_{k}.
	\end{align*}
	Since $t^{\rm TAS}_{k}>0$, the above equation implies that
	\begin{align*}
		\dfrac{1}{2}(L_{i^{*}}-L_{j^{*}})\|d^{k}\|^{2}t^{\rm TAS}_{k}+\langle\nabla f_{i^{*}}(x^{k})-\nabla f_{j^{*}}(x^{k}),d^{k}\rangle=0.
	\end{align*}
	Therefore,
	\begin{align*}
		t^{\rm TAS}_{k}=-\frac{2\langle\nabla f_{i^{*}}(x^{k})-\nabla f_{j^{*}}(x^{k}),d^{k}\rangle}{(L_{i^{*}}-L_{j^{*}})\|d^{k}\|^{2}}\in\mathcal{C}_{k}.
	\end{align*}
 	Combining the three cases completes the proof.
\end{proof}

\begin{remark}\label{rem:complexity}
	The candidate set $\mathcal{C}_k$ is finite. In particular, it contains
	at most $m$ individual minimizers and at most $\binom{m}{2}$ pairwise intersection points, in addition to the endpoint $t=1$. Hence, $|\mathcal{C}_k| \le 1 + m + \binom{m}{2}=
	O(m^2)$,
	where $m$ is the number of objective functions. Therefore, $t_k^{\rm TAS}$ can be computed exactly by evaluating $g_k(t)$ over the finite set $\mathcal{C}_k\cap[0,1]$ and selecting the point with the smallest function value. Although this process incurs a slightly higher computational cost at each iteration, this additional burden is usually negligible, as demonstrated by the numerical experiments in Section~\ref{num_exp}.
\end{remark}

Proposition \ref{prop:3.4} shows that the optimal solution of
\eqref{tight_ada_pro} can be obtained by evaluating $g_k(t)$ over the
finite candidate set $\mathcal{C}_k\cap[0,1]$. Therefore, the tight
adaptive step size can be computed without solving the one-dimensional
optimization problem \eqref{tight_ada_pro} by a generic numerical
optimization solver. Based on this observation, we summarize the
procedure for computing $t_k^{\rm TAS}$ in Algorithm \ref{alg:tas}.

\begin{algorithm}[H]
	\caption{Computation of the tight adaptive step size $t_k^{\rm TAS}$}
	\label{alg:tas}
	\begin{algorithmic}[1]
		\State \textbf{Input}: $\nabla f_{i}(x^{k})$ and $L_{i}$ for $i\in\langle m\rangle$, $d^{k}$
		\State \textbf{Output}: $t_k^{\rm TAS}$
		
		\State Initialize $\mathcal{C}_k\gets\{1\}$
		
		\For{$i=1$ \textbf{to} $m$}
		\State
		$t_i^*
		\gets
		-\dfrac{\langle\nabla f_i(x^k),d^k\rangle}
		{L_i\|d^k\|^2}$
		\If{$0<t_i^*<1$}
		\State
		$\mathcal{C}_k
		\gets
		\mathcal{C}_k\cup\{t_i^*\}$
		\EndIf
		\EndFor
		
		\For{$i=1$ \textbf{to} $m-1$}
		\For{$j=i+1$ \textbf{to} $m$}
		\If{$L_i\neq L_j$}
		\State
		$t_{ij}^*
		\gets
		-\dfrac{
			2\langle\nabla f_i(x^k)-\nabla f_j(x^k),d^k\rangle}
		{(L_i-L_j)\|d^k\|^2}$
		\If{$0<t_{ij}^*<1$}
		\State
		$\mathcal{C}_k
		\gets
		\mathcal{C}_k\cup\{t_{ij}^*\}$
		\EndIf
		\EndIf
		\EndFor
		\EndFor
		
		\State
		$t_k^{\rm TAS}
		\gets
		\underset{t\in\mathcal{C}_k}{\arg\min}\; g_k(t)$
		
	\end{algorithmic}
\end{algorithm}

\section{Convergence Analysis}\label{convergence}

In this section, we establish the convergence properties of the M-CondG
algorithm equipped with the proposed tight adaptive step size
$t_k^{\rm TAS}$. The analysis is based on the tight quadratic surrogate
$g_k$ introduced in \eqref{gk}. We first establish a descent estimate for the value of $g_k$ at the proposed step size.

\begin{proposition}\label{le:le0}
	Let $g_{k}$ be defined by \eqref{gk}. Consider the M-CondG algorithm with the tight adaptive step size. Then, there exists a constant $\widehat{L}\in(0,L_{\max}]$  such that
	\begin{equation}\label{gk_bounds}
		g_{k}(t_{k}^{\rm TAS})\leq-\frac{1}{2}\min\left\{-\theta(x^{k}),\frac{\theta(x^{k})^{2}}{ \widehat{L}D^2}\right\}.
	\end{equation}
\end{proposition}

\begin{proof}
	We consider the three possible forms of the optimal step size
	$t_k^{\rm TAS}$.
	
	\emph{Case 1.}
	Suppose that
	$t_k^{\rm TAS}=1$. Since $g_k(1)=\max_{i\in\langle m\rangle}g_{i,k}(1)$,
	there exists an index $i^\diamond\in\langle m\rangle$ such that
	\begin{equation}\label{tk1_1}
		g_k(t_k^{\rm TAS})
		=
		g_{i^\diamond,k}(1)
		=
		\frac{1}{2}L_{i^\diamond}\|d^k\|^2
		+
		\langle\nabla f_{i^\diamond}(x^k),d^k\rangle.
	\end{equation}
	Moreover, since $t_k^{\rm TAS}=1$ is the minimizer of the corresponding
	quadratic model on $[0,1]$, we have
	\begin{equation*}\label{tk1_2}
		-\frac{\langle\nabla f_{i^\diamond}(x^k),d^k\rangle}
		{L_{i^\diamond}\|d^k\|^2}
		\geq 1.
	\end{equation*}
	Hence, $L_{i^\diamond}\|d^k\|^2
	\leq-\langle\nabla f_{i^\diamond}(x^k),d^k\rangle$.
	Combining this inequality with \eqref{opt_val} and \eqref{tk1_1}, we obtain
	\begin{align}
		g_k(t_k^{\rm TAS})\leq
		\frac{1}{2}
		\langle\nabla f_{i^\diamond}(x^k),d^k\rangle 
		\leq
		\frac{1}{2}\theta(x^k).
		\label{tk1_3}
	\end{align}
	
	\emph{Case 2.}
	Suppose that $t_k^{\rm TAS}\in(0,1)$ coincides with the minimizer of
	one of the individual quadratic functions, say $g_{i^*,k}$, and that
	this function is active at $t_k^{\rm TAS}$. Then
	\[
	t_k^{\rm TAS}
	=
	-\frac{\langle\nabla f_{i^*}(x^k),d^k\rangle}
	{L_{i^*}\|d^k\|^2},
	\]
	and therefore
	\begin{equation*}\label{gi*tas}
		\begin{aligned}
			g_k(t_k^{\rm TAS})
			=
			g_{i^*,k}(t_k^{\rm TAS})
			=
			-\frac{
				\langle\nabla f_{i^*}(x^k),d^k\rangle^2
			}{
				2L_{i^*}\|d^k\|^2
			}.
		\end{aligned}
	\end{equation*}
	Since $\langle\nabla f_{i^*}(x^k),d^k\rangle
	\leq \theta(x^k)<0$
	and $\|d^k\|\leq D$, it follows that
	\begin{align}
		g_k(t_k^{\rm TAS})
		\leq
		-\frac{\theta(x^k)^2}
		{2L_{i^*}\|d^k\|^2}
		\leq
		-\frac{\theta(x^k)^2}
		{2L_{i^*}D^2}.
		\label{case2bound}
	\end{align}

	\emph{Case 3.}
	Suppose that $t_k^{\rm TAS}\in(0,1)$ is an intersection point of
	two or more active quadratic models. Let
	$i^\blacktriangle$ be any active index at $t_k^{\rm TAS}$. Then
	\begin{equation}\label{gi_star_tas}
		\begin{aligned}
			g_k(t_k^{\rm TAS})
			=
			g_{i^\blacktriangle,k}(t_k^{\rm TAS})
			=
			\frac{1}{2}L_{i^\blacktriangle}\|d^k\|^2
			(t_k^{\rm TAS})^2
			+
			\langle\nabla f_{i^\blacktriangle}(x^k),d^k\rangle
			t_k^{\rm TAS}.
		\end{aligned}
	\end{equation}
	The minimum value of $g_{i^\blacktriangle,k}(t)$ over $t\geq0$ is
	\[
	y^\blacktriangle
	=
	-\frac{
		\langle\nabla f_{i^\blacktriangle}(x^k),d^k\rangle^2
	}{
		2L_{i^\blacktriangle}\|d^k\|^2
	}.
	\]
	Since $y^\blacktriangle$ is the minimum value of
	$g_{i^\blacktriangle,k}$ and $t_k^{\rm TAS}$ is feasible, we have
	\[
	y^\blacktriangle
	\leq
	g_k(t_k^{\rm TAS}).
	\]
	On the other hand, by \eqref{gki_gk_hk} and the definition of
	$t_k^{\rm TAS}$,
	\[
	g_k(t_k^{\rm TAS})
	\leq
	g_k(t_k^{\rm AS})
	\leq
	h_k(t_k^{\rm AS})
	=
	-\frac{\theta(x^k)^2}
	{2L_{\max}\|d^k\|^2}.
	\]
	
	Consequently,
	\[
	-\frac{
		\theta(x^k)^2
	}{
		2L_{\max}\|d^k\|^2
	}
	\geq
	g_k(t_k^{\rm TAS})
	\geq
	-\frac{
		\langle\nabla f_{i^\blacktriangle}(x^k),d^k\rangle^2
	}{
		2L_{i^\blacktriangle}\|d^k\|^2
	}.
	\]
	Moreover, $
	|\langle\nabla f_{i^\blacktriangle}(x^k),d^k\rangle|
	\geq
	|\theta(x^k)|.
	$
	Hence, there exists a number
	$\omega_k\in[0,1]$ such that
	\begin{align}
		g_k(t_k^{\rm TAS})
		&=
		\omega_k
		\left(
		-\frac{\theta(x^k)^2}
		{2L_{\max}\|d^k\|^2}
		\right)
		+
		(1-\omega_k)
		\left(
		-\frac{
			\langle\nabla f_{i^\blacktriangle}(x^k),d^k\rangle^2
		}{
			2L_{i^\blacktriangle}\|d^k\|^2
		}
		\right)
		\notag\\
		&\leq
		-\frac{\theta(x^k)^2}
		{2\|d^k\|^2}
		\left(
		\frac{\omega_k}{L_{\max}}
		+
		\frac{1-\omega_k}{L_{i^\blacktriangle}}
		\right).
		\label{case3bound}
	\end{align}
	Define
	\begin{equation*}\label{Lhatk}
		\frac{1}{ L_k}
		=
		\frac{\omega_k}{L_{\max}}
		+
		\frac{1-\omega_k}{L_{i^\blacktriangle}}.
	\end{equation*}
	Since
	$L_{i^\blacktriangle}\leq L_{\max}$, we have $
	0< L_k\leq L_{\max}$.
	Let $\widetilde{L}=\sup\{ L_k:k\geq 0\}$. Then, we have $\widetilde L\in (0, L_{\max}]$. Using $\|d^k\|\leq D$ in \eqref{case3bound}, we obtain
	\[
	g_k(t_k^{\rm TAS})
	\leq
	-\frac{\theta(x^k)^2}
	{2 L_{k} D^2}\leq
	-\frac{\theta(x^k)^2}
	{2\widetilde L D^2}.
	\]
	
	Finally, let $\widehat{L}=\max\{\widetilde{L},L_{i^*}\}$. 
	Combining this with the above three cases gives
	\begin{equation*}
		g_{k}(t_{k}^{\rm TAS})\leq\frac{1}{2}\max\left\{\theta(x^{k}),-\frac{\theta(x^{k})^{2}}{ \widehat{L}D^2}\right\},
	\end{equation*}
	which completes the proof.
\end{proof}

In the following, we establish a basic result that characterizes the improvement of the functional values sequence $\{F(x^{k})\}$ generated by the M-CondG algorithm with the tight adaptive step size. This
implies that this sequence is componentwise decreasing.

\begin{proposition}\label{le:le1}
	Consider the M-CondG algorithm with the tight adaptive step size. Then there exists  $\widehat{L}\in(0,L_{\max}]$  such that
	\begin{equation}\label{eq:des1}
		F(x^{k}+t_{k}^{\rm TAS}d^{k})- F(x^{k})\preceq -\frac{1}{2}\min\left\{-\theta(x^{k}),\frac{\theta(x^{k})^{2}}{ \widehat{L}D^2}\right\}e.
	\end{equation}
\end{proposition}

\begin{proof}
	The result obviously follows from \eqref{descent-lemma} and \eqref{gk_bounds}.
\end{proof}

\begin{remark}
	Compared with Proposition \ref{tas_des}, the above proposition replaces $L_{\max}$ by a constant $\widehat{L}\in(0,L_{\max}]$.
	Consequently, the descent estimate obtained by $t_{k}^{\rm TAS}$ is never worse than that of $t_{k}^{\rm AS}$.
\end{remark}

As an application of Proposition \eqref{le:le1}, we establish the following convergence property of the MCondG algorithm with the adaptive step size strategy, without assuming convexity of the objective functions. Define 
\begin{equation*}
	f_{i}^{\inf}=\inf\{f_{i}(x):x\in\Omega\},\quad i\in\langle m\rangle.
\end{equation*}

\begin{theorem}\label{thm1}
	Consider the M-CondG algorithm with the tight adaptative step size. Let $i_{*}$ be an index such that $f_{i_{*}}(x^{0})-f_{i_{*}}^{\inf}=\min\{f_{i}(x^{0})-f_{i}^{\inf}:i\in\langle m\rangle\}$. Then there exists $\widehat{L}\in(0,L_{\max}]$ such that
	\begin{enumerate}[\rm(i)]
		\item $\lim_{k\rightarrow\infty}\theta(x^{k})=0$;
		\item for every $N\in\mathbb{N}$, there holds
		\begin{equation}\label{rate}
			\begin{aligned}
				\min&\{|\theta(x^k)|: k=0,1,\ldots,N-1\}\\
				&\leq
				\max\left\{
				\frac{2(f_{i_*}(x^0)-f_{i_*}^{\inf})}{N},
				\;
				D
				\sqrt{
					\frac{2\widehat{L}(f_{i_*}(x^0)-f_{i_*}^{\inf})}{N}
				}
				\right\}.
			\end{aligned}
		\end{equation}
	\end{enumerate}
\end{theorem}

\begin{proof}
	The proof is analogous to that of Corollary 14 in \cite{assunccao2021conditional}, and is therefore omitted.
\end{proof}

\begin{remark}
	{\rm (i)}
	A direct consequence of Theorem \ref{thm1}(i) and Lemma \ref{theta_property}(ii)--(iii) is that every limit point $\bar{x}$ of the sequence $\{x^k\}$ generated by the MCondG algorithm with the tight adaptive step size strategy is a Pareto critical point. Furthermore, by Proposition \ref{le:le1},
	\[
	f_i(\bar{x})=\inf\{f_i(x^k):k=0,1,\ldots\},\qquad \forall i\in\langle m\rangle.
	\]
	If, in addition, $F$ is convex on $\Omega$, then $\bar{x}$ is a weak Pareto optimal solution of problem \eqref{mop}.
	
	{\rm (ii)} Since $\widehat{L}\in(0,L_{\max}]$, the bound in \eqref{rate} is never worse than that of Corollary 14(ii) in \cite{assunccao2021conditional}. Moreover, if $\widehat{L}<L_{\max}$, the proposed adaptive step size provides a strictly tighter upper bound on
	\[
	\min_{0\le k\le N-1}|\theta(x^k)|,
	\]
	indicating an improved worst-case convergence guarantee.
\end{remark}

\section{Numerical experiments}\label{num_exp}

In this section, we present some numerical experiments to verify and demonstrate the performance of the proposed tight adaptive step size in comparison to the adaptive step size in \cite{assunccao2021conditional}. For convenience, we abbreviate the names of the two algorithms under comparison as follows:
\begin{itemize}
	\item \textbf{MCondG-AS}: the MCondG algorithm with $t_{k}=t_{k}^{\rm AS}$;
	\item \textbf{MCondG-TAS}: the MCondG algorithm with the tight adaptive step size obtained in Algorithm 1.
\end{itemize}

All codes are written in MATLAB R2020b and run on a PC with the 2.90 GHz Intel i7-10700 processor and 32 GB of RAM. For the subproblem \eqref{opt_sol}, we consider the equivalent formulation presented as follows:
\begin{equation}\label{linear_problem}
	\begin{aligned}
		\text{min}\quad &\tau\\
		\text{s.t.}\quad &\langle \nabla f_{i}(x),p-x \rangle\leq\tau,\quad i\in\langle m\rangle,\\
		& \tau\in\mathbb{R}, p\in\Omega.
	\end{aligned}
\end{equation}
In our experiments, we consider DGO2*, Toi8 and MGH33, which are listed in Example \ref{ex1}, \ref{ex2} and \ref{ex3}, respectively.
For each algorithm, we test each problem using 100 different initial points that are uniformly selected from the constrained set. The standard MATLAB subroutine \texttt{linprog} with default options is adopted to solve the problem \eqref{linear_problem}. The stopping criterion for the two algorithms is defined as $|\theta(x^{k})|\leq\epsilon$, where $\epsilon=10^{-5}$. The maximum number of allowed outer iterations is set to 1000, after which the algorithm is considered to have failed.

We compare MCondG-TAS and MCondG-AS using the following performance measurements: the average step size (Stp), the average number of iterations (Iter), and the average CPU time (in seconds) required to satisfy the stopping criterion. We additionally employ the \emph{Purity} and ($\Gamma$ and $\Delta$) \emph{Spread} metrics \cite{custodio2011direct} to quantitatively evaluate the quality and distribution of the obtained solution sets.

\subsection{Performance on test problems}

In this section, we compare MCondG-TAS with MCondG-AS on DGO2* and Toi8. The numerical results are reported in Table \ref{tab:numer_compare_dgo2-toi8}. For both problems, MCondG-TAS selects larger step sizes than MCondG-AS. The ratios of the mean step sizes are approximately 14.75 for DGO2* and 2.24 for Toi8, respectively. The larger step sizes also lead to fewer iterations and lower CPU times. For DGO2*, the average number of iterations decreases from 95.79 for MCondG-AS to 6.65 for MCondG-TAS, while the CPU time decreases from 0.67s to 0.06s. For Toi8, the corresponding numbers decrease from 146.89 to 92.69 iterations and from 1.01s to 0.66s, respectively. The reduction is particularly evident for DGO2*, whereas the difference between the two methods is smaller for Toi8.

The quality of the final solutions is comparable for the two methods, with some improvements obtained by MCondG-TAS. On DGO2*, the Purity value decreases from 2.0370 to 1.9643, and the $\Delta$-Spread decreases from 1.6859 to 1.4162. Both methods give the same $\Gamma$-Spread value of 0.7588. On Toi8, the two methods have the same Purity value of 2.0000. MCondG-TAS gives a slightly smaller $\Delta$-Spread, 1.3198 compared with 1.3602, and a smaller $\Gamma$-Spread, 2.1295 compared with 3.2995. Overall, the results indicate that MCondG-TAS can use larger step sizes without deteriorating the quality of the final solutions, while reducing the number of iterations and the CPU time.

\begin{table}[H]\footnotesize
	\centering
		\setlength{\tabcolsep}{4pt} 
	\caption{Numerical results obtained by MCondG-AS and MCondG-TAS on DGO2* and Toi8.}
	\footnotesize
	\resizebox{\linewidth}{!}{
	\begin{tabular}{c cccccc cccccc}
		\toprule
		& \multicolumn{6}{c}{MCondG-AS} & \multicolumn{6}{c}{MCondG-TAS} \\
		\cmidrule(lr){2-7} \cmidrule(lr){8-13}
		Problem & Stp &Iter & CPU & Purity & $\Delta$-Spread & $\Gamma$-Spread & Stp &Iter & CPU & Purity & $\Delta$-Spread & $\Gamma$-Spread \\
		\midrule
		DGO2*   & 2.4945e-3 & 95.79 & 0.67 & 2.0370 & 1.6859 & 0.7588 & 3.6806e-2 & 6.65 & 0.06 & 1.9643 & 1.4162 & 0.7586 \\
		Toi8   & 1.5817e-3 & 146.89 & 1.01 & 2.0000 & 1.3602 & 3.2995 & 3.5427e-3 & 92.69 & 0.66 & 2.0000 & 1.3198 & 2.1295 \\
		\bottomrule
	\end{tabular}
	}
	\label{tab:numer_compare_dgo2-toi8}
\end{table}

Since \(\theta(x^{k})\) characterizes the convergence behavior of the algorithm, we plot the evolution of \(|\theta(x^{k})|\) generated by MCondG-AS and MCondG-TAS using the initial point \(x^0\) given in Examples 3.1 and 3.2, as illustrated in Figure \ref{fig:theta_curve}. To visually highlight the convergence discrepancy between the two methods, the $y$-axis is set on a logarithmic scale. As observed from the subfigures, the curve corresponding to MCondG-TAS descends far more rapidly than that of MCondG-AS across both test problems. This means that MCondG-TAS can take larger steps while maintaining the descent property.
\begin{figure}[H]
	\centering
	\includegraphics[width=0.8\linewidth]{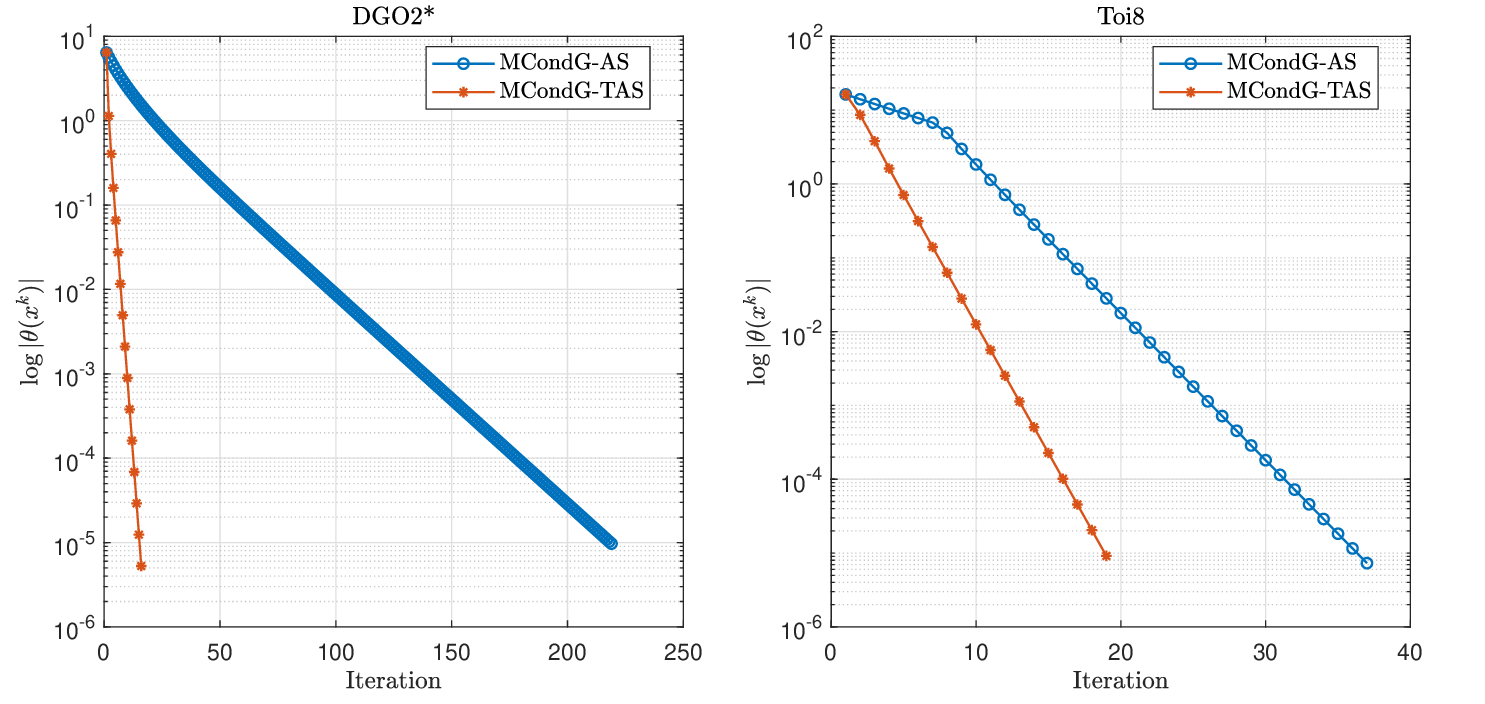} 
	\caption{Evolution of $\log|\theta(x^k)|$ generated by MCondG-AS and MCondG-TAS on DGO2* and Toi8.}
	\label{fig:theta_curve}
\end{figure}

\begin{example}\label{ex3}
	Consider the problem \eqref{mop} with
	\begin{align*}
		f_i(x) = \left(i\left(\sum_{j=1}^{n}jx_{j}\right)-1\right)^{2},\quad i\in\langle m\rangle
	\end{align*}
	and $\Omega=[-1,1]^{n}$. 
	This problem is referred to as MGH33, which can be found in \cite{mita2019nonmonotone,toint1983test}. 
\end{example}

For this example, direct calculation gives
\begin{equation*}
	L_{i}=\frac{i^{2}n(n+1)(2n+1)}{3},\quad i\in\langle m\rangle.
\end{equation*}
It is evident that the gradient Lipschitz constants associated with different objectives are distinct from one another. This property makes MGH33 an ideal benchmark for evaluating the performance of our proposed tight adaptive step size strategy. In our numerical tests, we adopt multiple combinations of parameters \(m\in\{2,3,5,10\}\) and \(n\in\{2,10,50\}\), which yield 12 specific test instances denoted as MGH33a to MGH33l. The numerical results for MCondG-AS and MCondG-TAS on these test cases are reported in Table \ref{tab:mgh33_compare}.

\begin{table}[H]
	\centering
	\setlength{\tabcolsep}{3pt} 
	\caption{Performance comparison of AS and TAS on MGH33 series problems}
	\label{tab:mgh33_comparison}
	\footnotesize
		\resizebox{\linewidth}{!}{
	\begin{tabular}{
			lcccccccccccccc
		}
		\toprule
		\multirow{2}{*}{Problem} &
		\multirow{2}{*}{$m$} &
		\multirow{2}{*}{$n$} &
		\multicolumn{6}{c}{MCondG-AS} &
		\multicolumn{6}{c}{MCondG-TAS} \\
		\cmidrule(lr){4-9} \cmidrule(lr){10-15}
		& & & {Stp} & {Iter} & {CPU} & {Purity} & {$\Delta$-Spread} & {$\Gamma$-Spread} &
		{Stp} & {Iter} & {CPU} & {Purity} & {$\Delta$-Spread} & {$\Gamma$-Spread} \\
		\midrule
		MGH33a & 2 & 2   & 1.0661e-2 & 21.91 & 0.15 & 5.4615 & 1.7377 & 0.1846 & 6.2944e-2 & 4.28 & 0.07 & 1.2241 & 1.3460 & 0.0967 \\
		MGH33b & 2 & 10  & 2.4464e-3 & 63.22 & 0.45 & 6.6667 & 1.9394 & 0.6842 & 9.4290e-3 & 14.57 & 0.12 & 1.1765 & 1.8372 & 0.5940 \\
		MGH33c & 2 & 50  & 7.0867e-4 & 105.54 & 0.73 & 5.0000 & 1.9798 & 1.0000 & 2.8979e-3 & 24.98 & 0.18 & 1.2500 & 1.9378 & 0.9790 \\
		MGH33d & 3 & 2   & 4.8116e-3 & 44.50 & 0.29 & 4.6111 & 1.7114 & 0.8811 & 7.2787e-2 & 3.45 & 0.06 & 1.2769 & 1.1901 & 0.2505 \\
		MGH33e & 3 & 10  & 1.0127e-3 & 133.39 & 0.86 & 30.0000 & 1.9798 & 4.0000 & 1.0854e-2 & 12.97 & 0.17 & 1.0345 & 1.7594 & 2.2898 \\
		MGH33f & 3 & 50  & 4.2274e-4 & 190.85 & 1.39 & 23.0000 & 1.9596 & 3.1786 & 3.5844e-3 & 21.60 & 0.15 & 1.0455 & 1.9104 & 3.5114 \\
		MGH33g & 5 & 2   & 1.6137e-3 & 109.61 & 0.71 & 3.8846 & 1.6879 & 4.0270 & 9.7424e-2 & 2.36 & 0.04 & 1.3467 & 1.1426 & 1.0559 \\
		MGH33h & 5 & 10  & 4.2415e-4 & 348.25 & 2.26 & 11.2500 & 1.9192 & 7.9438 & 1.4154e-2 & 10.31 & 0.13 & 1.0976 & 1.7520 & 6.5373 \\
		MGH33i & 5 & 50  & 1.0988e-4 & 618.38 & 4.43 & 45.0000 & 1.9798 & 16.0000 & 3.7135e-3 & 16.91 & 0.12 & 1.0227 & 1.9294 & 14.9648 \\
		MGH33j & 10 & 2   & 5.8221e-4 & 333.98 & 2.21 & 4.6500 & 1.6912 & 20.8633 & 1.1444e-1 & 2.03 & 0.04 & 1.2740 & 1.1664 & 7.2759 \\
		MGH33k & 10 & 10  & 2.0921e-4 & 644.54 & 6.05 & 23.5000 & 1.9733 & 58.6584 & 1.2701e-2 & 10.12 & 0.09 & 1.0444 & 1.8319 & 50.3431 \\
		MGH33l & 10 & 50  & 7.4283e-5 & 981.90 & 6.82 & 51.0000 & 29.3875 & 640.6558 & 3.8830e-3 & 18.39 & 0.25 & 1.0200 & 1.8854 & 69.8192 \\
		\bottomrule
	\end{tabular}
}\label{tab:mgh33_compare}
\end{table}

The average step sizes (Stp) obtained by MCondG-TAS are larger than those of MCondG-AS for all test instances. The gap between the two methods becomes larger as the number of objectives $m$ increases. For problems with substantially different gradient Lipschitz constants, the use of $L_{\max}$ in MCondG-AS leads to relatively small step sizes, whereas MCondG-TAS can select larger ones by taking the individual Lipschitz constants into account. The difference in step sizes is accompanied by a reduction in the number of iterations and CPU time. For all test instances, MCondG-TAS requires fewer iterations and less CPU time than MCondG-AS. For example, when $n=2$, the iteration count of MCondG-AS increases from $21.91$ to $333.98$ as $m$ increases from $2$ to $10$, while the corresponding iteration counts of MCondG-TAS are $4.28$, $3.45$, $2.36$, and $2.03$. The CPU time of MCondG-AS increases from $0.15$ s to $2.21$ s, whereas MCondG-TAS requires only $0.04$--$0.15$ s. The difference is also apparent for $n=50$. When $m=10$, for example, MCondG-AS requires $981.90$ iterations and $6.82$s, compared with $18.39$ iterations and $0.25$s for MCondG-TAS. The quality of the obtained solutions is also improved with MCondG-TAS. For all test instances, it gives smaller Purity, $\Delta$-Spread, and $\Gamma$-Spread values than MCondG-AS. For MGH33j, for example, the three values are $4.6500$, $1.6912$ and $20.8633$ for MCondG-AS, and $1.2740$, $1.1664$ and $7.2759$ for MCondG-TAS, respectively. For MGH33l, the corresponding values decrease from $51.0000$, $29.3875$ and $640.6558$ to $1.0200$, $1.8854$ and $69.8192$. Thus, the larger step sizes used by MCondG-TAS are not accompanied by a loss in solution quality.

The observed behavior is related to the way the Lipschitz constants are used in the two step size strategies. MCondG-AS uses the largest constant $L_{\max}$. When the objective-wise Lipschitz constants are far apart, this choice is governed by the largest one and consequently results in conservative step sizes. In MCondG-TAS, the individual constants $L_i$ are retained when constructing the candidate step sizes. The resulting step sizes are therefore less affected by $L_{\max}$, as reflected by the Stp values in Table~\ref{tab:mgh33_compare}. The difference becomes more evident as $m$ increases and is accompanied by a lower iteration count and CPU time. Although the TAS step size requires the evaluation of a larger set of candidates, the additional cost is small compared with the reduction in the number of iterations for these test problems. The MGH33 results therefore illustrate the advantage of using objective-wise Lipschitz constants when their values differ substantially across the objectives.

To investigate the convergence behavior of the proposed MCondG-TAS  in comparison with MCondG-AS on MGH33, we performed experiments on instances MGH33a--MGH33l. All algorithms were initialized from the same point, \(x^0 = (0.5, 0.5, \dots, 0.5)^{\top}\). Figure \ref{fig:mgh33theta} presents the evolution of \(\log|\theta(x^{k})|\) versus the number of iterations for both methods.  As is evident from the plots, MCondG-TAS exhibits a remarkably fast convergence, consistently driving the the stationarity measure below the tolerance threshold of $10^{-5}$ within typically fewer than 20 iterations on all 12 problems. This observation underscores the strong robustness of MCondG-TAS against increases in problem dimensionality. Conversely, MCondG-AS follows a relatively smooth log-linear trajectory of descent but requires a significantly larger number of iterations to satisfy the same stopping criterion. More critically, for instances with $m=10$, MCondG-AS fails to reach the tolerance of $10^{-5}$ 
within the upper limit of permissible iterations.

\begin{figure}[H]
	\centering
	\includegraphics[width=\linewidth]{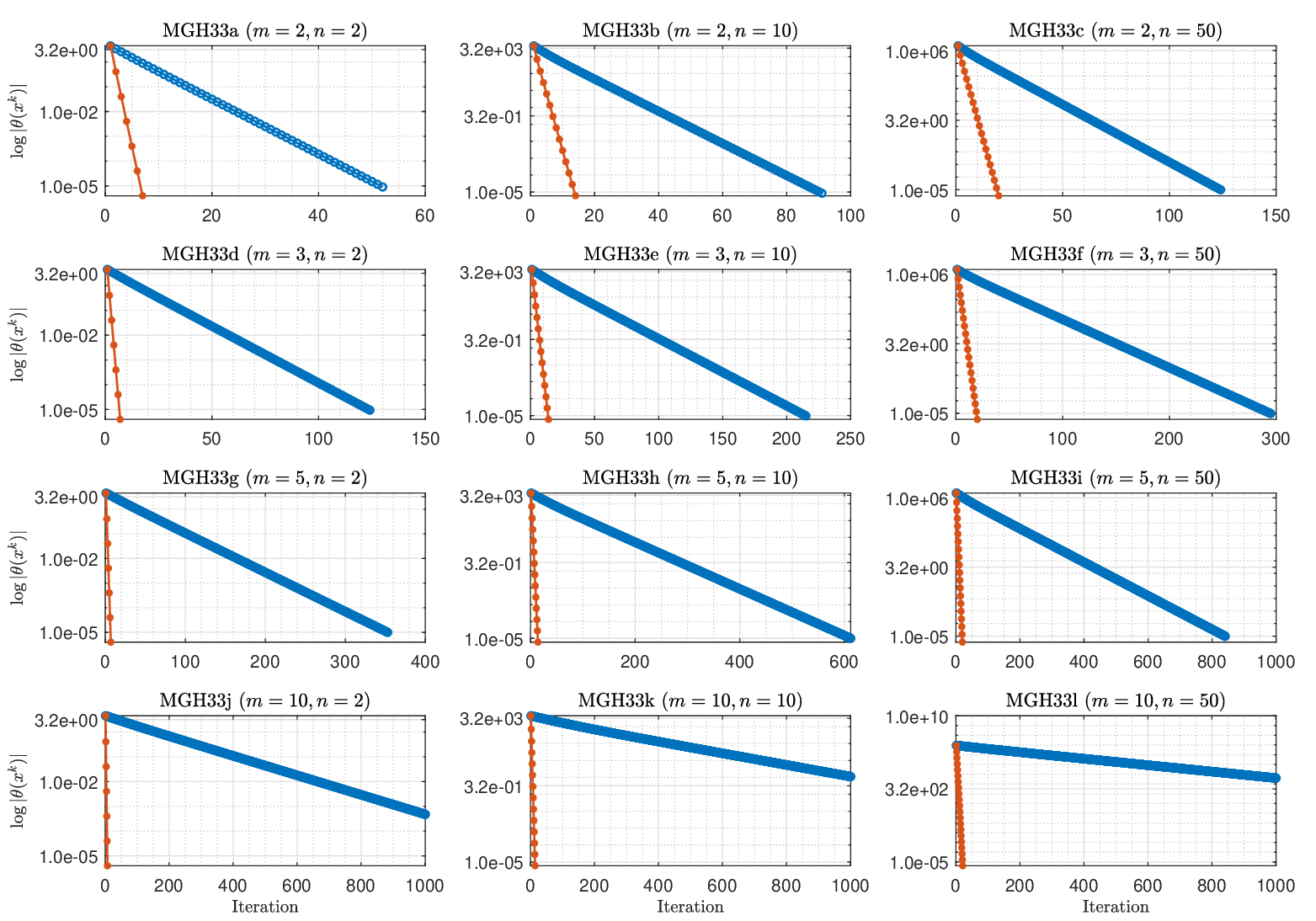}
	\caption{Convergence behavior of $|\theta(x^k)|$ generated by MCondG-AS (blue) and MCondG-TAS (orange) on MGH33 with varying dimensions $(m,n)$.}
	\label{fig:mgh33theta}
\end{figure}

In order to have a good comparison, we plot the final solutions obtained by MCondG-AS and MCondG-TAS for six test problems, namely, DGO2*, Toi8, MGH33a, MGH33d, MGH33g and MGH33j. Figure \ref{fig:pff1} presents the image sets of DGO2*, Toi8, MGH33a, and MGH33d, together with the final solutions obtained by the two algorithms. Since the number of objectives in MGH33g and MGH33j exceeds three, parallel coordinate plots are employed for visualization. Figure \ref{fig:pff2} shows the final solutions obtained by MCondG-AS and MCondG-TAS for these two problems. As can be observed from Figures \ref{fig:pff1} and \ref{fig:pff2}, MCondG-TAS generally obtains solutions that provide a better approximation of the Pareto front than those obtained by MCondG-AS.

\begin{figure}
	\centering
	\includegraphics[width=0.23\linewidth]{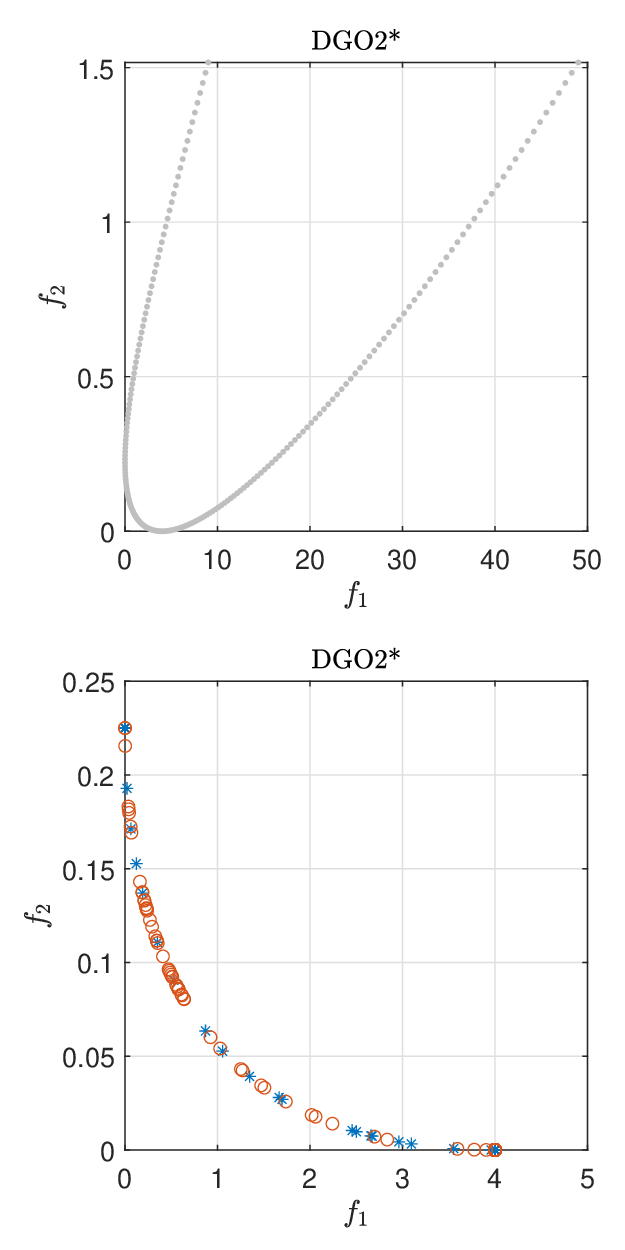}\hskip0.1in
	\includegraphics[width=0.23\linewidth]{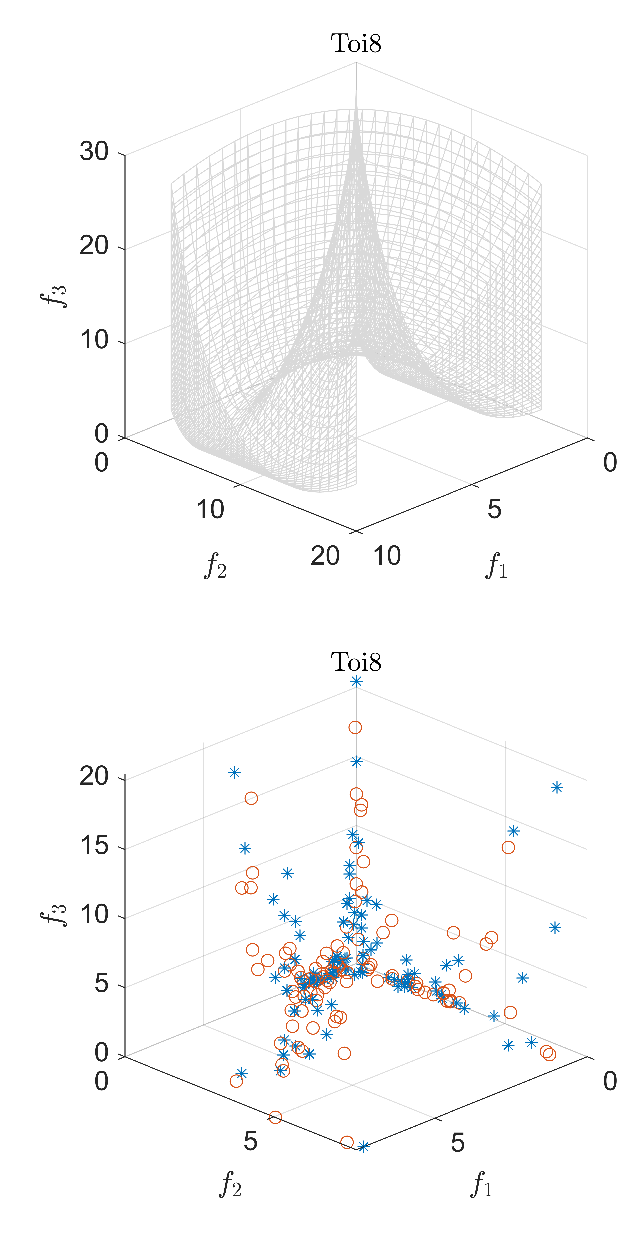}\hskip0.1in
	\includegraphics[width=0.23\linewidth]{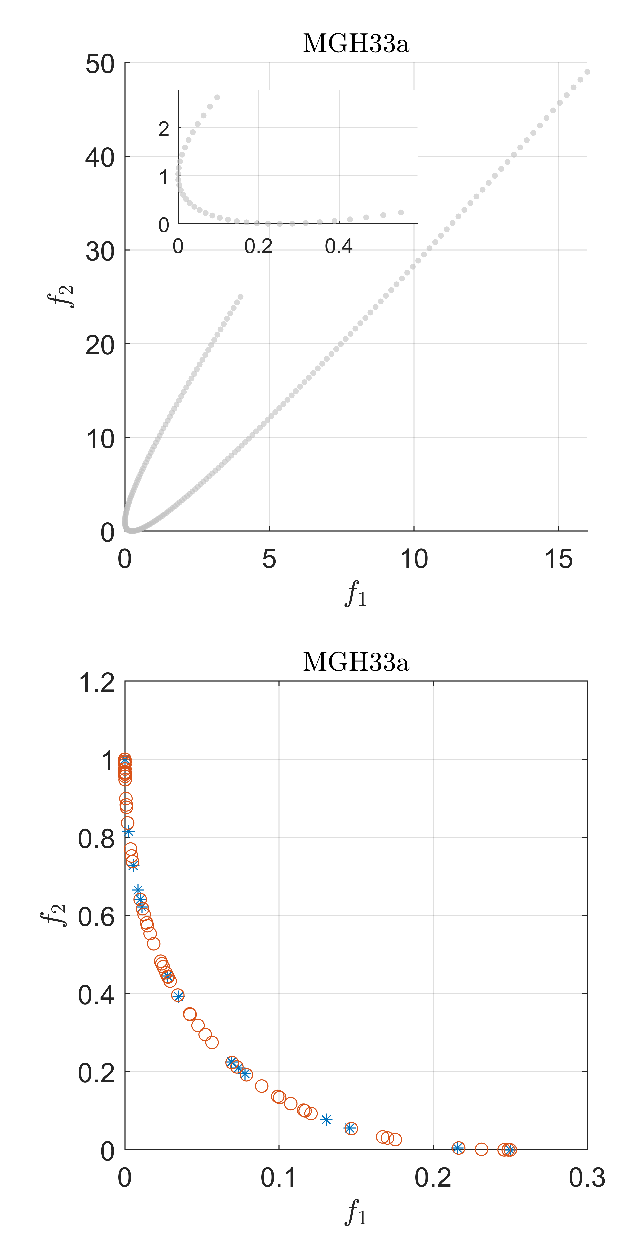}\hskip0.1in
	\includegraphics[width=0.23\linewidth]{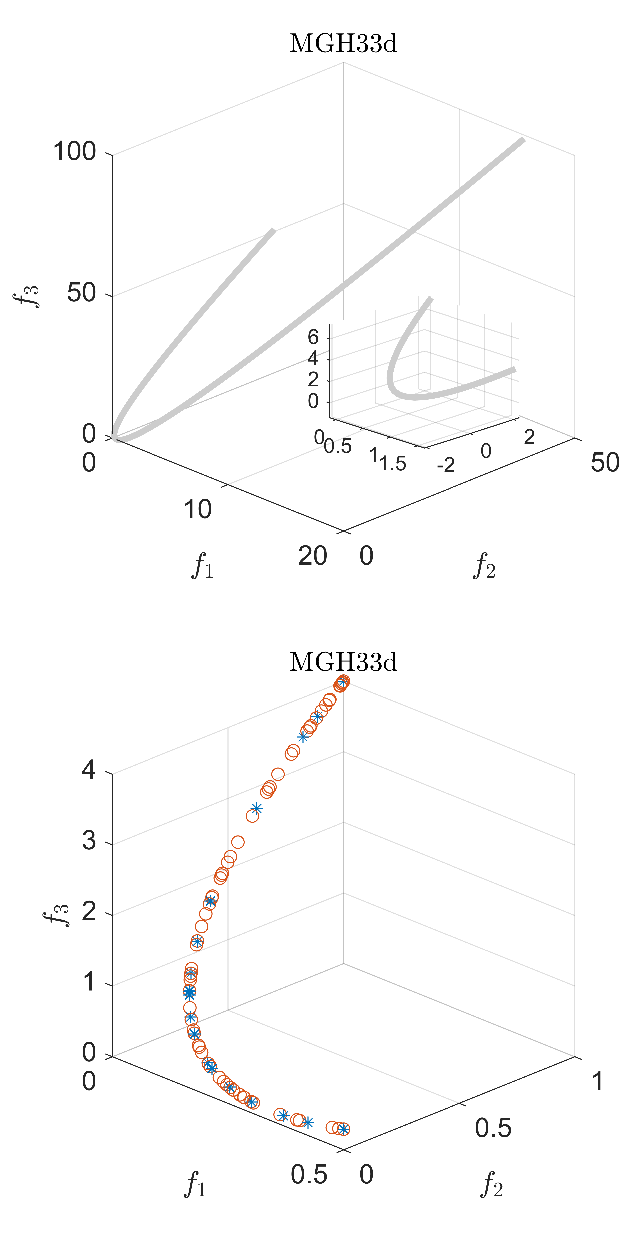}
	\caption{Image sets and final solutions of MCondG‑AS (light‑blue asterisks) and MCondG‑TAS (orange‑red circles) for \(\text{DGO2}^*\), Toi8, MGH33a and MGH33d.}
	\label{fig:pff1}
\end{figure}

\begin{figure}[H]
	\centering
	\includegraphics[width=0.6\linewidth]{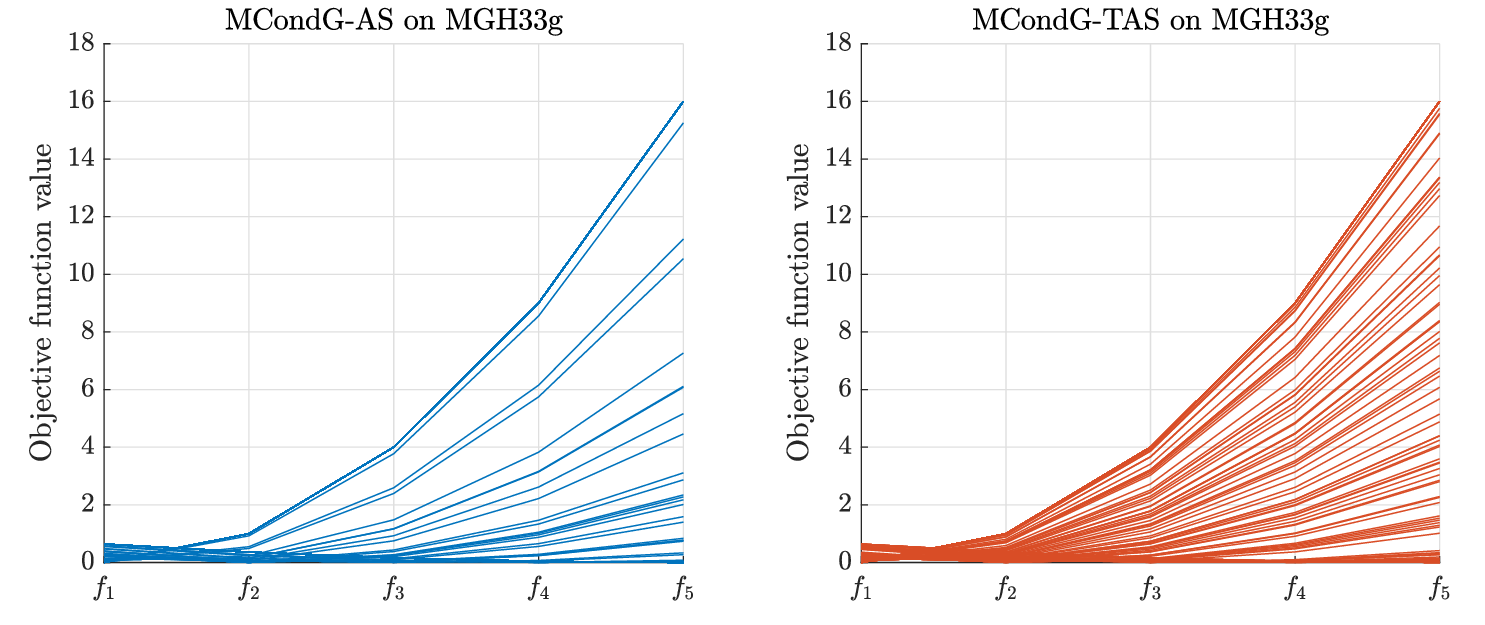}
	\includegraphics[width=0.6\linewidth]{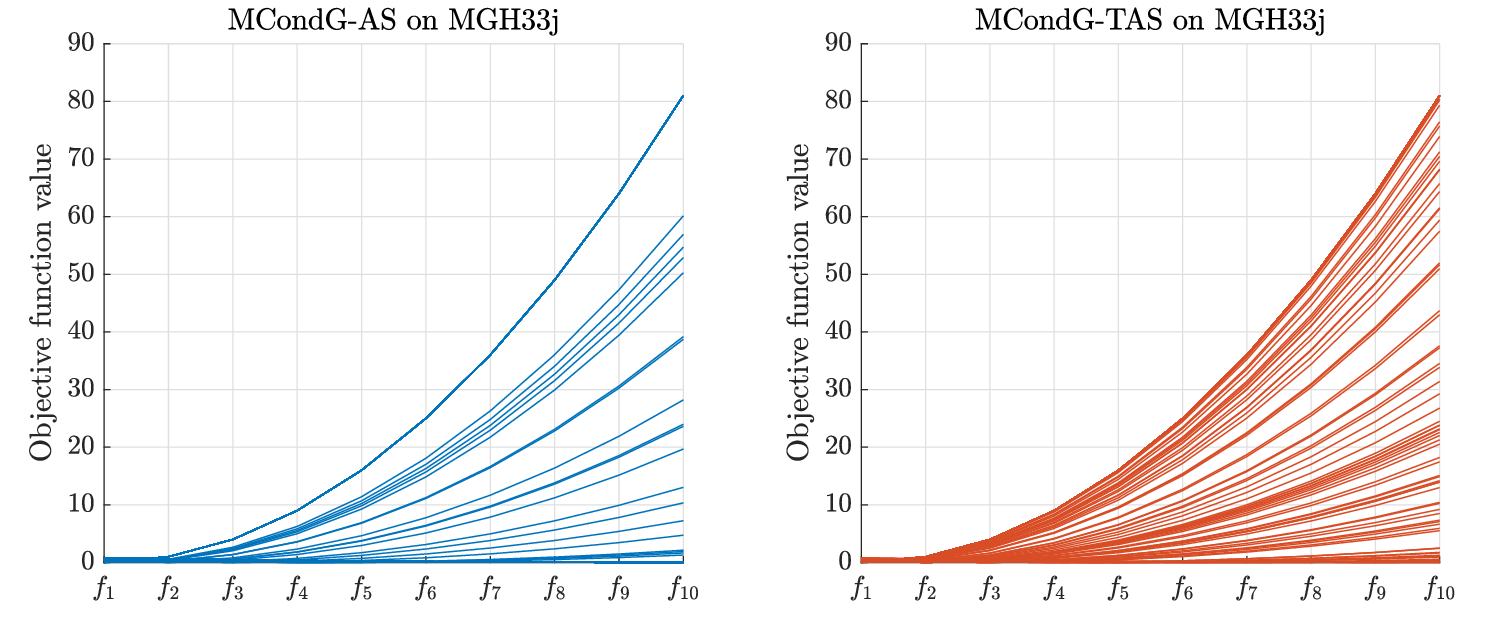}
	\caption{Parallel‑coordinate plots of the final solutions obtained by MCondG‑AS and MCondG‑TAS on MGH33g (\(m=5\)) and MGH33j (\(m=10\)).}
	\label{fig:pff2}
\end{figure}

Overall, the numerical results confirm that the proposed tight adaptive step size effectively alleviates the conservativeness of the original adaptive step size. By exploiting individual curvature information of different objective functions, MCondG-TAS can take larger steps while maintaining the descent property and solution quality, resulting in faster convergence and lower computational costs.

\subsection{Performance on real-world problem}

To further evaluate the practical performance of the proposed tight
adaptive step size, we consider a real-world multiobjective optimization problem arising from the design and optimization of a packed bed latent
heat thermal energy storage system (PBLHTS). Such thermal storage systems have potential applications in solar power generation, industrial waste heat recovery, heating, ventilation, air conditioning and other energy-related applications. 

Following Gao et al.~\cite{gao2020multi}, three performance criteria
are considered simultaneously, namely, the effective heat storage time
$t_{\rm eff}$, the effective heat storage capacity $Q_{\rm eff}$, and
the exergy efficiency $\varphi_{\rm ex}$. Accordingly, the problem can
be formulated as the following three-objective optimization problem:
\begin{equation*}
	\min_{x\in\Omega} F(x)
	=
	\left(f_1(x),f_2(x),f_3(x)\right)^{\top},
\end{equation*}
where
\begin{equation*}
	f_1(x)=t_{\rm eff},\qquad
	f_2(x)=-Q_{\rm eff},\qquad
	f_3(x)=-\varphi_{\rm ex}.
\end{equation*}
The negative signs in $f_2$ and $f_3$ convert the maximization of the
effective heat storage capacity and the exergy efficiency into
minimization objectives. The decision vector consists of nine thermophysical and geometric
design parameters of the thermal storage system, and the feasible set $\Omega$ is a box-constrained set. The complete mathematical expressions of the three objective functions,
together with the detailed definitions and physical interpretations of the decision variables and their bounds, can be found in
Gao et al. \cite{gao2020multi}.  An important feature of this problem is that the three objective
functions generally exhibit heterogeneous curvature characteristics.
Accordingly, their gradient Lipschitz constants are different, i.e.,
$L_{i}=\sup_{x\in\Omega}\|\nabla^{2}f_{i}(x)\|_{2}$ for $i=1,2,3$. 

The results in Table \ref{tab:realmop} demonstrate the effectiveness of
MCondG-TAS on the PBLHTS problem. The average step size of MCondG-TAS is
 approximately $8.73$ times that of MCondG-AS, reflecting the benefit of retaining the individual
gradient Lipschitz constants rather than relying solely on $L_{\max}$. The
larger step sizes substantially reduce the computational cost: MCondG-TAS
requires only $80.50$ iterations and $1.21$ s of CPU time, compared with
$714.70$ iterations and $8.68$ s for MCondG-AS, corresponding to reductions
of approximately $88.7\%$ and $86.1\%$, respectively. Meanwhile, MCondG-TAS
achieves better values for all three solution-quality indicators, with the
Purity, $\Delta$-Spread, and $\Gamma$-Spread decreasing by approximately
$33.4\%$, $3.6\%$, and $10.2\%$, respectively. As shown in Figure~\ref{fig:realmoppff},
the final solutions obtained by the two methods are distributed along a
similar region of the image set, indicating that the improved computational
efficiency of MCondG-TAS is not achieved at the expense of solution quality.

\begin{table}[H]\footnotesize
		\setlength{\tabcolsep}{18pt} 
	\centering
	\caption{Numerical results of MCondG-AS and MCondG-TAS on PBLHTS.}
	\label{tab:comparison}
	\begin{tabular}{ccccccc}
		\toprule
		Algorithm & Stp & Iter & CPU & Purity & $\Delta$-Spread & $\Gamma$-Spread \\
		\midrule
		MCondG-AS  & 7.4326e-3 & 714.70 & 8.68 & 2.4925 & 0.7772 & 548956 \\
	   MCondG-TAS & 6.4928e-2 & 80.5  & 1.21  & 1.6600  & 0.7496 & 492854 \\
		\bottomrule
	\end{tabular}
	\label{tab:realmop}
\end{table}

\begin{figure}[H]
	\centering
	\includegraphics[width=0.8\linewidth]{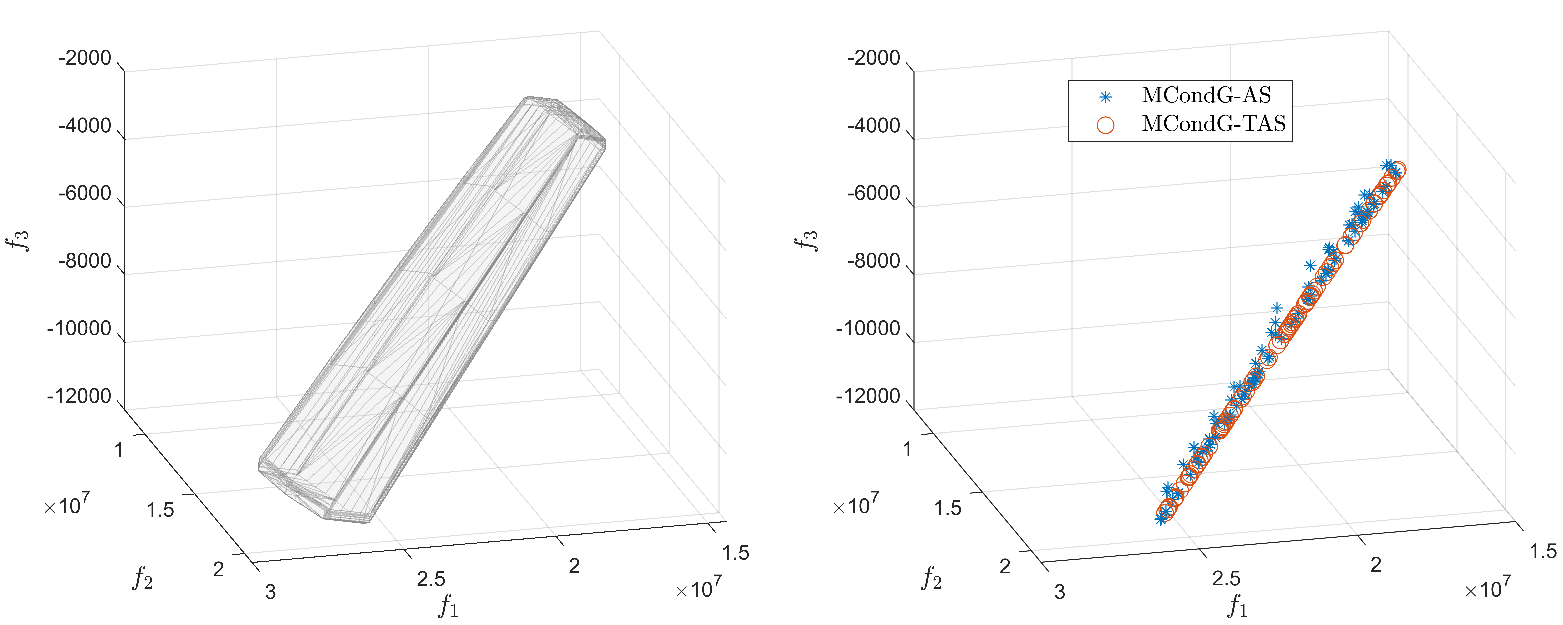}
	\caption{Image sets (left) and the final solutions (right) obtained by MCondG-AS and MCondG-TAS on PBLHTS.}
	\label{fig:realmoppff}
\end{figure}

\section{Concluding remarks and discussions}\label{conclusion}

In this paper, we proposed a tighter surrogate-based adaptive step size strategy for the multiobjective conditional gradient method to solve constrained multiobjective optimization problems with Lipschitz smooth objectives. The proposed algorithm modifies the classical adaptive step size scheme by constructing a tighter piecewise quadratic surrogate function that explicitly incorporates the individual Lipschitz constants of the objective gradients. Asymptotic global convergence to Pareto stationary points and a tighter worst-case sublinear convergence bound of the proposed MCondG-TAS algorithm were established under standard gradient Lipschitz assumptions.

Several promising directions warrant future investigation. 
Integrating the proposed tight adaptive step size with advanced conditional gradient variants, such as away-step Frank--Wolfe methods \cite{gonccalves2024away} and the method in \cite{gonccalves2025improved}, is expected to achieve further acceleration. 
The current method assumes that the gradient Lipschitz constants are available. Since the Lipschitz constants \(L_i\) are generally unknown or expensive to compute, developing a dynamic estimation strategy for each $L_i$ within the algorithmic framework is a practically relevant direction.

\section*{References}


\end{document}